\documentclass[11pt]{amsart}

\usepackage{amssymb,amscd,amsmath,hyperref,color,enumerate}
\usepackage[all,cmtip]{xy}
\usepackage{mathrsfs}
\usepackage{lipsum}
\usepackage{nicematrix,tikz, ifthen}
\usepackage{scalerel}
\usepackage{chemmacros}
\usepackage{ascii}


\usepackage[margin=3cm]{geometry}

\newcommand{\N}{{\ensuremath{\mathbb{N}}}}

\newcommand{\C}{{\ensuremath{\mathbb{C}}}}

\usepackage[normalem]{ulem}

\newcommand{\stkout}[1]{\ifmmode\text{\sout{\ensuremath{#1}}}\else\sout{#1}\fi}

\DeclareMathOperator{\spc}{sp}

\DeclareMathOperator{\Max}{Max}

\DeclareMathOperator{\GL}{GL}

\DeclareMathOperator{\U}{U}
\DeclareMathOperator{\SU}{SU}
\DeclareMathOperator{\SL}{SL}

\DeclareMathOperator{\diag}{diag}
\DeclareMathOperator{\rad}{rad}

\DeclareMathOperator{\Ad}{Ad}

\newcommand{\ca}[1]{\ensuremath{\mathcal{#1}}}

\newtheorem{proposition}{Proposition}[section]
\newtheorem{lemma}[proposition]{Lemma}

\newtheorem{theorem}[proposition]{Theorem}

\newtheorem{corollary}[proposition]{Corollary}
\theoremstyle{definition}

\newtheorem{caseinner}{Case}

\newenvironment{case}
{%
	\caseinner
}
{\endcaseinner}

\counterwithin*{subcase}{caseinner}

\newtheorem{remark}[proposition]{Remark}

\numberwithin{equation}{section}

\newlength{\leftstackrelawd}
\newlength{\leftstackrelbwd}
\def\leftstackrel#1#2{\settowidth{\leftstackrelawd}%
	{${{}^{#1}}$}\settowidth{\leftstackrelbwd}{$#2$}%
	\addtolength{\leftstackrelawd}{-\leftstackrelbwd}%
	\leavevmode\ifthenelse{\lengthtest{\leftstackrelawd>0pt}}%
	{\kern-.5\leftstackrelawd}{}\mathrel{\mathop{#2}\limits^{#1}}}

\begin{document}
	
	\title[Continuous spectrum-shrinking maps between finite-dimensional algebras]{Continuous spectrum-shrinking maps between finite-dimensional algebras}
	
	\author{Ilja Gogi\'{c}, Mateo Toma\v{s}evi\'{c}}
	
	\address{I.~Gogi\'c, Department of Mathematics, Faculty of Science, University of Zagreb, Bijeni\v{c}ka 30, 10000 Zagreb, Croatia}
	\email{ilja@math.hr}
	
	\address{M.~Toma\v{s}evi\'c, Department of Mathematics, Faculty of Science, University of Zagreb, Bijeni\v{c}ka 30, 10000 Zagreb, Croatia}
	\email{mateo.tomasevic@math.hr}
	
	\thanks{We thank Peter \v{S}emrl for pointing us to the relevant literature on Kaplansky-type problems.}
	

	
	\date{\today}
	
	\begin{abstract}
		Let $\mathcal{A}$ and $\mathcal{B}$ be unital finite-dimensional complex algebras, each equipped with the unique Hausdorff vector topology. Denote by  $\mathrm{Max}(\mathcal{A})=\{\mathcal{M}_1, \ldots, \mathcal{M}_p\}$ and 
		$\mathrm{Max}(\mathcal{B})=\{\mathcal{N}_1, \ldots, \mathcal{N}_q\}$ the sets of all maximal ideals of $\mathcal{A}$ and $\mathcal{B}$, respectively. For each $1 \leq i \leq p$ and $1 \leq j \leq q$ define the quantities
		$$k_i:=\sqrt{\dim(\mathcal{A}/\mathcal{M}_i)} \quad \text{ and } \quad  m_j:=\sqrt{\dim(\mathcal{B}/\mathcal{N}_j)},$$
		which are positive integers by Wedderburn's structure theorem. We show that there exists a continuous spectrum-shrinking map $\phi: \mathcal{A} \to \mathcal{B}$ (i.e.\ $\mathrm{sp}(\phi(a))\subseteq \mathrm{sp}(a)$ for all $a \in \mathcal{A}$) if and only if for each $1\leq j \leq q$ the linear Diophantine equation
		$$
		k_1x_{1}^{j} + \cdots + k_px_{p}^j = m_j
		$$
		has a non-negative integer solution $(x_{1}^j,\ldots,x_{p}^j)\in \mathbb{N}_{0}^p$. In a similar manner we also characterize the existence of continuous spectrum-preserving maps $\phi: \mathcal{A} \to \mathcal{B}$ (i.e.\  $\mathrm{sp}(\phi(a))= \mathrm{sp}(a)$ for all $a \in \mathcal{A}$). Finally, we analyze conditions under which all continuous spectrum-shrinking maps  $\phi: \mathcal{A} \to \mathcal{B}$ are automatically spectrum-preserving.
	\end{abstract}

	\date{\today}
	
	\maketitle

	\section{Introduction}  
	A Jordan homomorphism between complex (associative) algebras $\ca{A}$ and $\ca{B}$ is a linear map $\phi: \ca{A}\to \ca{B}$ which preservers squares, i.e.\
	\begin{equation*}
		\phi(a^2) = \phi(a)^2, \quad \text{ for all } a \in \ca{A},
	\end{equation*}
	or equivalently, it satisfies the condition
	\begin{equation*}
		\phi(ab+ba) = \phi(a)\phi(b) + \phi(b)\phi(a), \quad \text{ for all } a,b \in \ca{A}.
	\end{equation*}
	Typical examples of Jordan homomorphisms include linear multiplicative and antimultiplicative maps. A central problem, with a rich historical context, is to identify conditions on algebras $\ca{A}$ and $\ca{B}$ under which any (typically surjective) Jordan homomorphism $\phi: \ca{A} \to \ca{B}$   is either multiplicative or antimultiplicative, or more generally, can be expressed as  a suitable combination of these types. Foundational results on this topic can be found in \cite{Herstein, JacobsonRickart, Smiley}, and for more recent developments, see \cite{Bresar} and the references therein.
	
	\smallskip

	The study of Jordan homomorphisms is of particular importance in the theory of Banach algebras. It is well-known (and easy to verify) that any unital Jordan homomorphism $\phi : \ca{A} \to \ca{B}$ between unital algebras $\ca{A}$ and $\ca{B}$ preserves invertibility, meaning $\phi(a)$ is invertible whenever $a \in \ca{A}$ is invertible (see e.g.\ \cite[Proposition~1.3]{Sourour}). Moreover, $\phi$ is  spectrum-shrinking, i.e.\ $\spc(\phi(a)) \subseteq \spc(a)$ for all $a \in \ca{A}$. A major open problem in this area is the Kaplansky-Aupetit question, which asks whether  Jordan epimorphisms between unital semisimple Banach algebras can be characterized as surjective linear spectrum-shrinking maps \cite{Aupetit, Kaplansky}. This question has received considerable attention, with progress made in certain special cases, but it remains unsolved, even for $C^*$-algebras (see e.g.\ \cite[p.~270]{BresarSemrl}).  Furthermore, in many situations it is more convenient to deal with spectrum-preserving maps (i.e.\ $\spc(\phi(a)) = \spc(a)$ for all $a \in \ca{A}$). When the map $\phi: \ca{A} \to \ca{B}$ is linear and unital, note that $\phi$ is spectrum-preserving if and only if it preserves invertibility in both directions (i.e.\ $a \in \ca{A}$ is invertible if and only if  $\phi(a)$ is invertible). In particular, if a result for spectrum-preserving maps is established, it is natural to inquire whether it can be extended to spectrum-shrinking maps. However, the literature suggests that such improvements are generally far from simple. For instance, spectrum-preserving linear surjections between algebras of bounded linear maps on Banach spaces were characterized in \cite[Theorem~1]{JafarianSourour}, with a more concise proof provided in \cite[Theorem~2]{Semrl}.  In contrast, deriving an analogous result for spectrum-shrinking maps \cite[Theorem 1.1]{Sourour} is much more challenging, requiring a significantly longer proof involving complex analysis. Similarly, the techniques used to characterize spectrum-shrinking linear maps between matrix algebras in \cite{Pazzis, RodmanSemrl} are much more intricate than those for spectrum-preserving maps (as demonstrated in \cite{Semrl}).

	In our recent preprint \cite{ChirvasituGogicTomasevic1}, we observed that in certain cases continuous (not necessarily linear) spectrum-shrinking maps are automatically spectrum-preserving. For instance, let $M_n=M_n(\C)$ denote the algebra of $n \times n$ complex matrices and let $\ca{X}_n\subseteq M_n$ be one of the following subsets: $M_n$ itself, the general linear group $\GL(n)$, the special linear group $\SL(n)$, the unitary group $\U(n)$, or the subset of $n \times n$ normal matrices. Then, by \cite[Corollary~1.2]{ChirvasituGogicTomasevic1}, for an arbitrary $m \in \N$ there exists a continuous spectrum-shrinking map $\phi : \ca{X}_n \to M_m$ if and only if $n$ divides $m$, and in that case $\phi$ is necessarily spectrum-preserving. On the other hand, the analogous result fails for the special unitary group $\SU(n)$ and the space of self-adjoint matrices (see \cite[Remarks~1.4]{ChirvasituGogicTomasevic1}). A more general statement  that encompasses  \cite[Corollary~1.2]{ChirvasituGogicTomasevic1} is stated in \cite[Theorem~1.1]{ChirvasituGogicTomasevic1}, with a further refinement for general connected compact groups given in \cite{Chirvasitu}. Additionally, a variant of \cite[Theorem~1.1]{ChirvasituGogicTomasevic1}, concerning maps on singular matrices, is presented in  \cite[Theorem~1.2]{ChirvasituGogicTomasevic2}.
	
	\smallskip 
	
	The purpose of this note is to establish a variant of \cite[Theorem~1.1]{ChirvasituGogicTomasevic1} for maps between arbitrary unital finite-dimensional complex algebras. Before stating our main result, first recall that, according to  Molien's theorem (see e.g.\ \cite[Corollary~2.66]{Bresar-book}), any simple finite-dimensional complex algebra is isomorphic to a full matrix algebra $M_n$ for some $n\in \N$. In particular, if $\ca{A}$ is a finite-dimensional complex algebra, then for each $\ca{M} \in \Max(\mathcal{A})$ (the set of all maximal ideals of $\ca{A}$), the quantity $\sqrt{\dim(\mathcal{A}/\mathcal{M})}$ is a positive integer. We also recall that a \emph{structural matrix algebra (SMA)}  is a unital subalgebra of $M_n$ ($n \in \N$) linearly  spanned by a set of matrix units \cite{VanWyk}. Equivalently, SMAs are precisely the subalgebras of $M_n$ that contain the diagonal matrices (see \cite[Proposition 3.1]{GogicTomasevic1}).
	\begin{theorem}\label{thm:main} Let $\mathcal{A}$ and $\ca{B}$ be unital finite-dimensional complex algebras, each equipped with the unique Hausdorff vector topology. If $\mathrm{Max}(\ca{A})=\{\ca{M}_1, \ldots, \ca{M}_p\}$ and 
		$\mathrm{Max}(\ca{B})=\{\ca{N}_1, \ldots, \ca{N}_q\}$,  define the quantities
		\begin{equation}\label{eq:k_i}
			k_i:=\sqrt{\dim(\mathcal{A}/\mathcal{M}_i)} \quad \text{ and } \quad  m_j:=\sqrt{\dim(\mathcal{B}/\mathcal{N}_j)} \quad (1 \leq i \leq p, \, 1 \leq j \leq q).
		\end{equation}
		Then the following holds:
		\begin{itemize}
			\item[(a)] There exists a continuous spectrum-shrinking map $\phi: \ca{A} \to \ca{B}$ if and only if for each $1\leq j \leq q$ the linear Diophantine equation
			\begin{equation}\label{eq:diophantine equation 1}
				k_1x_{1}^{j} + \cdots + k_px_{p}^j = m_j
			\end{equation}
			has a non-negative integer solution $(x_{1}^j,\ldots,x_{p}^j)\in \mathbb{N}_{0}^p$.
			\item[(b)]There exists a continuous spectrum-preserving map $\phi: \ca{A} \to \ca{B}$ if and only if there exists a family $\{(x_{1}^j,\ldots,x_{p}^j) : 1 \leq j \leq q\}$ of non-negative integer solutions to \eqref{eq:diophantine equation 1}, with the property that for each $1 \leq i \leq p$ there exists some  $1 \leq j \leq q$ with $x_i^j>0$.
			\item[(c)]  If every continuous spectrum-shrinking map $\phi: \mathcal{A} \to \mathcal{B}$ is spectrum-preserving, then any family $\{(x_{1}^j,\ldots,x_{p}^j) : 1 \leq j \leq q\}$ of non-negative integer solutions to \eqref{eq:diophantine equation 1} satisfies that for each $1 \leq i \leq p$ there exists some  $1 \leq j \leq q$ with $x_i^j>0$. The converse holds when $\mathcal{A}$ is isomorphic to an SMA.
		\end{itemize} 
	\end{theorem}
	
	The proof of Theorem \ref{thm:main} will be provided in the next section. At the moment we are uncertain whether the converse of Theorem \ref{thm:main} (c) holds for arbitrary  unital finite-dimensional complex algebras $\ca{A}$ and we anticipate to investigate it in future work.
	
	\section{Proof of Theorem \ref{thm:main}}
	Before proving Theorem \ref{thm:main}, we introduce some notation that will be used throughout. Let $\ca{A}$ be a unital finite-dimensional (associative) complex algebra with identity $1_\ca{A}$. We assume that $\ca{A}$ is equipped with the unique Hausdorff vector topology. As already noted, by $\Max(\ca{A})$ we denote the set of all maximal ideals of $\ca{A}$. The \emph{(Jacobson) radical} of $\ca{A}$ is denoted by $\rad(\ca{A})$. Since $\ca{A}$ is finite-dimensional, it follows that $\rad(\ca{A})$ is the largest nilpotent ideal of $\ca{A}$ (see e.g.\ \cite[Section~2.1]{Bresar-book} and \cite[Section~3.1]{DrozdKirichenko-book}).  Moreover, by Wedderburn's structure theorem (\cite[Section~2.9]{Bresar-book} or \cite[Section~2.4]{DrozdKirichenko-book}), we also have
	\begin{equation}\label{eq:radical}
		\rad(\ca{A})=\bigcap\{\ca{M} : \ca{M} \in \Max(\ca{A})\}. 
	\end{equation}
	Furthermore, as the underlying field $\C$ is perfect, we can invoke Wedderburn's principal theorem (see e.g.\ \cite[p.~191]{Rowen-book}), which guarantees the existence of a semisimple subalgebra $\ca{A}_{ss}$ of $\ca{A}$, isomorphic to $\ca{A}/\rad({\ca{A}})$, such that $\ca{A}$ is  the vector space direct sum of $\ca{A}_{ss}$ and $\rad(\ca{A})$. Throughout, we denote by:
	\begin{itemize}
		\item $Q_{\ca{A}} : \ca{A} \to \ca{A}/\rad(\ca{A})$, the canonical quotient map,  
		\item $\Pi_{\ca{A}} : \ca{A} \to \ca{A}_{ss}$, the idempotent with  image $\ca{A}_{ss}$ and kernel $\rad(\ca{A})$.
	\end{itemize}
	
	\smallskip 
	
	Next, by $\C[x]$ we denote the polynomial algebra in one variable $x$  over $\C$. For $a \in \ca{A}$, by $\C[a]$ we denote the unital subalgebra of $\ca{A}$ generated by $a$, i.e. 
	$$
	\C[a]=\{p(a) : p \in \C[x]\} \subseteq \ca{A}.
	$$ 
	Let  $\ca{A}^\times$ denote the group of all invertible elements in $\ca{A}$. The spectrum of an element $a\in \ca{A}$ is denoted by $\spc(a)$ (or $\spc_{\ca{A}}(a)$ when we want to emphasize the underlying algebra), and is defined as
	$$
	\spc(a):=\{\lambda \in \C : \lambda 1_{\ca{A}} - a   \notin \ca{A}^\times\}.
	$$
	Since $\ca{A}$ is finite-dimensional, every element $a\in \ca{A}$ is algebraic (i.e.\ there exists a nonzero $p \in \C[x]$ such that $p(a)=0$),  implying that the spectrum of any element in $\ca{A}$ is nonempty and finite.  Furthermore, the inverse of any invertible element $a\in \ca{A}$ lies in the subalgebra $\C[a]$. Thus, if $\ca{B}$ is any unital finite-dimensional complex algebra with $1_{\ca{B}}=1_{\ca{A}}$,  the spectrum of any element $a\in \ca{A}\cap \ca{B}$  is the same whether computed in $\ca{B}$ or $\ca{A}$. Finally, note that $\mathcal{A}^{\times}$ is path-connected. Indeed, for any $a \in \mathcal{A}^\times$, the finiteness of the spectrum allows the selection of a suitable branch of the logarithm, yielding $a = \exp(b)$ for some $b \in \mathcal{A}$ (via the  holomorphic functional calculus \cite[Section~3.3]{KadisonRingrose-book}). Then 
	$$
	[0,1] \to \ca{A}^\times, \quad t \mapsto \exp(tb),
	$$ defines a continuous path from $1_{\ca{A}}$ to $a$ within $\mathcal{A}^\times$. 
	
	\smallskip

	We state the following straightforward fact and provide its proof for completeness, as we were unable to locate a direct reference.
	\begin{lemma}\label{le:specradical}
		Let $\ca{A}$ be a unital finite-dimensional complex algebra. Then, for all $a \in \ca{A}$ we have
		$$
		\spc_{\ca{A}}(a)=\spc_{\ca{A}/\rad(\ca{A})}(Q_{\ca{A}}(a))=\spc_{\ca{A}_{ss}}(\Pi_{\ca{A}}(a)).
		$$
	\end{lemma}
	\begin{proof}
		Let $a \in \ca{A}$ be fixed. It suffices to show that 
		$$a \text{ is invertible in } \ca{A}  \iff Q_{\ca{A}}(a) \text{ is invertible in } \ca{A}/\rad(\ca{A}).
		$$
		The implication $\Longrightarrow$ is trivial, so assume that  $ Q_{\ca{A}}(a)$ is invertible in $\ca{A}/\rad(\ca{A})$. As $\rad(\ca{A})$ is a nilpotent ideal, there exists $b \in \ca{A}$ and a nilpotent element $z \in \ca{A}$ such that
		\begin{equation}\label{eq:ab}
			ab=1_{\ca{A}}-z.
		\end{equation}
		Since $1_{\ca{A}}-z\in \ca{A}^\times$ and  $\ca{A}$ is finite-dimensional, \eqref{eq:ab} implies that $a\in \ca{A}^\times$. 
	\end{proof}

	\smallskip
	
	In the sequel, for integers $k<l$ by $[k,l]$ we denote the set of all integers between $k$ and $l$.
	
	\begin{proof}[Proof of Theorem \ref{thm:main}]
		First of all, note that it suffices to prove the theorem when $\ca{B}$ is a semisimple algebra. Indeed, by Wedderburn's principal theorem we have $\ca{B}=\ca{B}_{ss} \dotplus \rad(\ca{B})$ (as vector spaces, where $\ca{B}_{ss}$ is a subalgebra of $\ca{B}$ isomorphic to $\ca{B}/\rad({\ca{B}})$). Then there exists a continuous spectrum-shrinking  (respectively, spectrum-preserving) map $\ca{A} \to \ca{B}$ if and only if  there exists a continuous spectrum-shrinking  (respectively, spectrum-preserving) map $\ca{A} \to \ca{B}_{ss}$. Indeed, let $\phi : \ca{A} \to \ca{B}$ be an arbitrary continuous spectrum-shrinking (preserving) map.  By the established notation and Lemma \ref{le:specradical}, the map
		$\Pi_{\ca{B}} \circ \phi :  \ca{A} \to \ca{B}_{ss}$ retains the same spectral property (and is clearly continuous). The converse implication is trivial. An analogous argument also shows that all continuous spectrum-shrinking maps $\ca{A} \to \ca{B}$ are spectrum-preserving if and only if the same condition holds for maps  $\ca{A} \to  \ca{B}_{ss}$. In light of Wedderburn’s structure theorem, we may therefore assume without loss of generality that
		$$\ca{B} = M_{m_1} \oplus \cdots \oplus M_{m_q},$$
		and regard $\ca{B}$ as a block-diagonal subalgebra of $M_m$ in a natural way, where $m:=m_1 + \cdots + m_q$. Under this identification, any map $\phi: \ca{A} \to \ca{B}$  decomposes as
		$$\phi= (\phi_1,\ldots,\phi_q),$$ 
		with each $\phi_j : \ca{A} \to M_{m_j}$ denoting the corresponding coordinate map. It is then clear that $\phi$ is spectrum-shrinking if and only if each $\phi_j$ is spectrum-shrinking, for all $1 \le j \le q$.
		
		\begin{case}
			First assume that $\ca{A}$ is isomorphic to an SMA.
		\end{case}
		As the entire statement of Theorem \ref{thm:main} is invariant under isomorphism, we may assume without loss of generality that $\mathcal{A}\subseteq M_n$ is an SMA for some $n \in \mathbb{N}$.
		
		\smallskip
		
		Assume that $\phi : \ca{A} \to  \ca{B}=M_{m_1} \oplus \cdots \oplus M_{m_q}$ is a continuous spectrum-shrinking map. Then, as noted before, each component $\phi_j : \ca{A} \to M_{m_j}$ is a continuous spectrum-shrinking map.
		
		By \cite[Section 2]{Akkurt} (see also \cite[Lemma~3.2]{GogicTomasevic1}\label{le:AkkurtSMA}), there exists a permutation matrix $R \in M_n^\times$ such that
		\begin{equation}\label{eq:Akkurt permutation argument}
			R \ca{A} R^{-1}=\begin{bmatrix} M_{k_1} & M_{k_1,k_2}^\dagger & \cdots & M_{k_1,k_p}^\dagger \\ 0 & M_{k_2} & \cdots & M_{k_2,k_p}^\dagger \\
				\vdots & \vdots & \ddots & \vdots \\
				0 & 0 & \cdots & M_{k_p} \end{bmatrix}
		\end{equation}
		for some $p, k_1,\ldots,k_p \in \N$ such that $k_1+\cdots+k_p = n$, where for any $1\leq i < j \leq n$, $M_{k_i,k_j}^\dagger$ is either zero or the space of all $k_i \times k_j$ rectangular complex matrices. Once again, since the full statement of Theorem \ref{thm:main} is invariant under isomorphism (and in particular under conjugation), we may without loss of generality assume that $\mathcal{A}$ is already in the form given on the right-hand side of \eqref{eq:Akkurt permutation argument}. In that case we have  
		$$\rad(\ca{A})=\begin{bmatrix} 0_{k_1} & M_{k_1,k_2}^\dagger & \cdots & M_{k_1,k_p}^\dagger \\ 0 & 0_{k_2} & \cdots & M_{k_2,k_p}^\dagger \\
			\vdots & \vdots & \ddots & \vdots \\
			0 & 0 & \cdots & 0_{k_p} \end{bmatrix},$$
		so that the quantities $k_1, \ldots, k_p$ of \eqref{eq:k_i} are precisely the corresponding sizes of the diagonal blocks (see \cite[Proposition 4.1]{Coelho} for a more intrinsic characterization of the radical of SMAs).
		
		\smallskip
		
		Denote by $\ca{D}_n$ the subalgebra of all diagonal matrices in $M_n$. As already noted, $\mathcal{D}_n \subseteq \mathcal{A}$. Set 
		$$\Lambda_n := \diag(1,\ldots,n) \in \ca{D}_n.$$ 
		Given a matrix $X$, denote by 
		$$k_X(x) := \det(x I-X)$$ 
		its characteristic polynomial. Fix some $1 \le j \le q$. We have
		\begin{equation}\label{eq:initial choice of js}
			k_{\phi_j(\Lambda_n)}(x) = (x-1)^{\ell_1(j)}\cdots (x-n)^{\ell_n(j)}
		\end{equation}
		for some $\ell_1(j),\ldots,\ell_n(j) \ge 0$ such that $\ell_1(j)+\cdots+\ell_n(j) = m_j$.
		
		\smallskip
		
		Denote by $\ca{E}_n \subseteq \ca{D}_n$ the set of all diagonal matrices with $n$ distinct diagonal entries. It is easy to see that $\ca{E}_n$ is path-connected. Consider the continuous map
		$$F : \ca{A}^\times \times \ca{E}_n \to \C[x]_{\le m_j}, \quad F(S,D) := k_{\phi_j(SDS^{-1})},$$
		where $\C[x]_{\le m_j}$ is the subspace of $\C[x]$ consisting of polynomials of degree $\le m_j$, endowed with the unique Hausdorff vector topology. For each $D = \diag(\lambda_1,\ldots,\lambda_n) \in \ca{E}_n$ and $S \in \ca{A}^\times$ there exist $\ell_1(j,S,D), \ldots, \ell_n(j,S,D) \ge 0$ with sum $m_j$, such that
		$$k_{\phi_j(SDS^{-1})}(x) = (x-\lambda_1)^{\ell_1(j,S,D)}\cdots (x-\lambda_n)^{\ell_n(j,S,D)}.$$
		For any $(d_1,\ldots,d_n) \in \N_0^n$ such that $d_1+\cdots+d_n = m_j$, consider the continuous map
		$$F_{(d_1,\ldots,d_n)} : \ca{A}^\times \times \ca{E}_n \to \C[x]_{\le m_j}, \quad F_{(d_1,\ldots,d_n)}(S,\diag(\lambda_1,\ldots,\lambda_n)) := (x-\lambda_1)^{d_1}\cdots (x-\lambda_n)^{d_n}.$$
		Clearly, these functions map each element of $\ca{A}^\times \times \ca{E}_n$ into distinct polynomials. By \cite[Lemma~1.7]{ChirvasituGogicTomasevic1}, it follows that 
		$F$ equals exactly one of them. Since $F(I,\Lambda_n) = F_{(\ell_1(j),\ldots,\ell_n(j))}(I,\Lambda_n)$ by \eqref{eq:initial choice of js}, it follows $F = F_{(\ell_1(j),\ldots,\ell_n(j))}$ and therefore
		\begin{equation}\label{eq:before permuting}
			k_{\phi_j(S\diag(\lambda_1,\ldots,\lambda_n)S^{-1})}(x) = (x-\lambda_1)^{\ell_1(j)}\cdots (x-\lambda_n)^{\ell_n(j)}
		\end{equation}
		for all $S \in \ca{A}^\times$ and (distinct, and hence by density arbitrary) $\lambda_1,\ldots,\lambda_n \in \C$.
		
		\smallskip
		
		Let $1 \le l \le p$ and fix some distinct $r,s \in [k_1+\cdots + k_{l-1}+1, k_1+\cdots + k_l]$ (we formally set $k_0 := 0$). The permutation matrix $P \in M_n^\times$ pertaining to the permutation which swaps $r$ and $s$ is contained in $\ca{A}^\times$ by the assumption. Note that $P^{-1}\Lambda_n P$ is obtained from $\Lambda_n$ by swapping $r$ and $s$. Therefore,
		\begin{align*}
			\prod_{1 \le i \le n} (x-i)^{\ell_i(j)} &= F(I,\Lambda_n) = F(P, P^{-1}\Lambda_n P) \\
			&= \left(\prod_{i 
				\in [1,n]\setminus \{r,s\}} (x-i)^{\ell_i(j)}\right)(x-r)^{\ell_s(j)}(x-s)^{\ell_r(j)}.
		\end{align*}
		It follows $\ell_r(j) = \ell_s(j)$. Overall, we conclude
		\begin{align*}
			&\ell_1(j) = \cdots = \ell_{k_1}(j), \\
			&\ell_{k_1+1}(j)=\cdots = \ell_{k_1+k_2}(j), \\
			&\qquad \vdots\\
			&\ell_{k_1+\cdots + k_{p-1}+1}(j) = \cdots= \ell_{k_1+\cdots + k_{p}}(j).
		\end{align*}
		In particular, we have
		$$
		m_j = \ell_1(j)+\cdots + \ell_n(j) = k_1\ell_{k_1}(j) + \cdots + k_p \ell_{k_1+\cdots + k_{p}}(j)
		$$
		so 
		\begin{equation}\label{eq:final family}
			(x_1^j, \ldots, x_p^j):=(\ell_{k_1}(j),\ldots,\ell_{k_1+\cdots + k_{p}}(j)) \in \N_0^p
		\end{equation} is a solution to \eqref{eq:diophantine equation 1}. This shows the ``$\implies$'' implication of (a).
		
		\smallskip
		
		Also note that, by \eqref{eq:before permuting}, for each $S \in \ca{A}^\times$ and $\lambda_1,\ldots,\lambda_n \in \C$, we have
		\begin{align}\label{eq:final exponents}
			k_{\phi(S\diag(\lambda_1,\ldots,\lambda_n)S^{-1})}(x) &= \prod_{1 \le j \le q} k_{\phi_j(S\diag(\lambda_1,\ldots,\lambda_n)S^{-1})}(x) \nonumber \\
			&= \prod_{1 \le j \le q}(x-\lambda_1)^{\ell_1(j)}\cdots (x-\lambda_n)^{\ell_n(j)} \nonumber \\ 
			&= \prod_{1 \le i \le n} (x-\lambda_i)^{\sum_{1 \le j \le q}\ell_i(j)}.
		\end{align}
		We now address the ``$\implies$'' part of (b). Suppose, in addition, that $\phi$ is spectrum-preserving. Then, by \eqref{eq:final exponents}, the associated family of solutions given in \eqref{eq:final family} must satisfy $\sum_{1 \le j \le q}\ell_i(j) > 0$ for all $1 \le i \le n$. In particular, the family  
		$$
		\{(\ell_{k_1}(j),\ldots,\ell_{k_1+\cdots + k_{p}}(j)) : 1 \le j \le q\}
		$$
		meets the requirement stated in part (b).
		
		\smallskip
		
		We now prove the ``$\impliedby$'' direction of (c). Suppose that for each family $\{(x_{1}^j,\ldots,x_{p}^j) : 1 \le j \le q\}$ of non-negative integer solutions of \eqref{eq:diophantine equation 1} holds that for all $1 \leq i \leq p$ there exists some  $1 \leq j \leq q$ with $x_i^j>0$. In particular, if there exists some continuous spectrum-shrinking map $\phi : \ca{A}\to \ca{B}$, then the associated family of solutions derived in \eqref{eq:final family} satisfies the above condition, so that $\sum_{1 \le j \le q}\ell_i(j) > 0$ for all $1 \le i \le n$. As $\ca{A}$ is an SMA, by \cite[Lemma 3.4]{GogicTomasevic2} the set
		$$\Ad_{\ca{A}^\times}(\ca{E}_n)=\{S\diag(\lambda_1,\ldots,\lambda_n)S^{-1} : S\in \ca{A}^\times, \lambda_1,\ldots,\lambda_n\in \C \text{ pairwise distinct}\}$$
		is dense in $\ca{A}$. Therefore, by \eqref{eq:final exponents} and the continuity of $\phi$ we conclude that $\phi$ spectrum-preserving. 
		
		\smallskip
		
		We now prove the remaining directions in (a), (b) and (c). Suppose that \eqref{eq:diophantine equation 1} has a solution $(x_1^j,\ldots,x_p^j) \in \N^p_0$ for each $1 \le j \le q$. Define the map $\phi_j : \ca{A} \to M_{m_j}$ by sending $X \in \ca{A}$ with diagonal blocks $X_1\in M_{k_1}, \ldots, X_p \in M_{k_p}$ to the block-diagonal matrix in $M_{m_j}$, where $X_i$ appears on the diagonal exactly $x_i^j$ times (the blocks can be arranged in any order).  It is clear that 
		$$
		\phi := (\phi_1,\ldots,\phi_q) : \ca{A} \to M_{m_1} \oplus \cdots \oplus M_{m_q}=\ca{B}
		$$ 
		is a continuous spectrum-shrinking map, which finishes the ``$\impliedby$'' part of (a). Moreover, note that $\phi$ fails to be spectrum-preserving if and only if there exists $1 \le i \le p$ such that $x_i^j = 0$, for all $1\le j \le q$. The latter establishes the ``$\impliedby$'' direction of (b) and the ``$\implies$'' direction of (c).

		\begin{case}
			Now assume that $\ca{A}$ is a general unital finite-dimensional algebra (while $\ca{B} = M_{m_1} \oplus \cdots \oplus M_{m_q}$ remains semisimple).
		\end{case}
		\smallskip
		
		(a) $\&$ (b). Note that there exists a continuous spectrum-shrinking  (spectrum-preserving)  map $\ca{A} \to  \ca{B}$ if and only if there exists a map $\ca{A}_{ss} \to  \ca{B}$ with the same property. Indeed, if $\phi : \ca{A} \to \ca{B}$ is a continuous spectrum-shrinking   (spectrum-preserving) map, then its restriction to $\ca{A}_{ss}$ inherits the same  properties. Conversely, suppose that $\psi : \ca{A}_{ss} \to \ca{B}$ is a continuous spectrum-shrinking  (spectrum-preserving)  map. Then, in view of Lemma \ref{le:specradical}, the map
		$\psi \circ \Pi_{\ca{A}} : \ca{A}  \to \ca{B}$
		satisfies the same properties as $\psi$. Therefore, it suffices to establish the claim when $\ca{A}$ is already semisimple and of the form $\ca{A} = M_{k_1} \oplus \cdots \oplus M_{k_p}$. This, in turn, follows directly from Case 1, as $M_{k_1} \oplus \cdots \oplus M_{k_p}$ can be viewed as an SMA in $M_n$ in a natural way, where $n = k_1 + \cdots + k_p$. This completes the proof of (a) and (b).
		
		\smallskip
		
		(c). Similarly as in the SMA case, assume that \eqref{eq:diophantine equation 1} admits a family  $\{(x_1^j,\ldots,x_p^j) \in \N^p_0 : 1 \le j \le q\}$ of solutions such that there exists $1 \le i \le p$ with $x_i^j = 0$, for all $1\le j \le q$. By similar arguments as in the ``$\implies$'' direction of (c) of Case 1, after the identification $\ca{A}_{ss} = M_{k_1} \oplus \cdots \oplus M_{k_p}$, we obtain a continuous spectrum-shrinking map $\phi : \ca{A}_{ss} \to \ca{B}$ which is not spectrum-preserving. Then $\phi \circ \Pi_{\ca{A}} : \ca{A} \to \ca{B}$ is the desired map. The proof of the theorem is now complete.
	\end{proof}
	
	\begin{corollary}
		Let $\ca{A}$ be a unital finite-dimensional complex algebra. Then $\ca{A}$ admits a continuous ``eigenvalue selection'' (i.e.\ a spectrum shrinker $\ca{A} \to \C$) if and only if $\ca{A}$ contains an ideal of codimension one.
	\end{corollary}
	\begin{proof}
		The Diophantine equation \eqref{eq:diophantine equation 1} has a non-negative integer solution for $m=1$ if and only if one of the numbers $k_1,\ldots,k_p$ is equal to $1$. This is equivalent to the fact that $\ca{A}$ contains a (maximal) ideal of codimension one.
	\end{proof}

	\begin{remark}
		Note that for a semisimple algebra $\ca{A}$, the ``spectrum-shrinking $\implies$ spectrum-preserving'' statement of Theorem \ref{thm:main} (b) does not directly follow from \cite[Theorem 1.1]{ChirvasituGogicTomasevic1}, unless the algebra $\ca{A}$ is already simple. Indeed,  assume that $\ca{A} = M_{k_1} \oplus \cdots \oplus M_{k_p} \subseteq M_n$, where $n := k_1 + \cdots + k_p$. While all conditions of \cite[Theorem 1.1]{ChirvasituGogicTomasevic1} are satisfied, the group $\ca{A}^\times \cap S_n$ (where $S_n$ is the symmetric group, identified with the $n \times n$ permutation matrices) \emph{does not} act transitively on $[1,n]$. Instead, $\ca{A}^\times \cap S_n$ consists of permutations that fix each of the sets $[1,k_1], [k_1+1,k_2], \ldots, [k_1 + \cdots + k_{p-1} + 1, k_1 + \cdots + k_p]$.
	\end{remark}

	\begin{remark}
		If $p \ge 2$ and the numbers $k_i$'s of \eqref{eq:k_i} are coprime, then the largest $m \in \N$ for which there is no continuous spectrum-shrinking map $\phi: \ca{A} \to M_m$ is precisely the Frobenius number $g(k_1, \ldots, k_p)$ (see e.g.\ \cite{Alfonsin}).
	\end{remark}

\end{document}